\documentclass[12pt]{amsart}

\newtheorem{theorem}{Theorem}[section]

\newtheorem{lemma}[theorem]{Lemma}
\theoremstyle{proposition}

\newtheorem{corollary}[theorem]{Corollary}
\usepackage[top=1 in, bottom=1 in, left=1 in, right=1 in]{geometry}
\newtheorem{definition}[theorem]{Definition}
\newtheorem{example}[theorem]{Example}

\theoremstyle{remark}

\usepackage[colorinlistoftodos]{todonotes}

\newtheorem{remark}[theorem]{Remark}

\numberwithin{equation}{section}

\usepackage{times}
\usepackage{enumerate}
\usepackage{graphicx,adjustbox}
\usepackage{hyperref}
\usepackage{amsmath,amssymb}
\usepackage{amscd}
\usepackage{graphicx}
\usepackage[all]{xy}
\usepackage{mathtools}

\usepackage{xcolor}

\title[On Certain Gluing of semigroup rings and indispensable resolution of semigroup rings]
{On Certain Gluing and indispensable resolution of semigroup rings}
\author{Sanjay Kumar Singh \and Pranjal Srivastava}
\date{}
\address{\small \rm  } 
\address{\small \rm  Discipline of Mathematics, IISER Bhopal, Madhya Pradesh, India.}
\email{sanjayks@iiserb.ac.in}
\thanks{The second author thanks IISER Bhopal for the Institute post-doc fellowship IISERB/DoFA/PDF/2023/80.}
\address{\small \rm  Discipline of Mathematics, IISER Bhopal, Madhya Pradesh, India.}
\email{pranjal.srivastava194@gmail.com}
\date{}



\begin{document}
	
	\begin{abstract}
	In this paper, our aim is twofold: First, by using the technique of gluing semigroups, we
	give infinitely many families of a projective closure with the Cohen-Macaulay (Gorenstein) property. Also, we give an
	effective technique for constructing large families of $1$-dimensional Gorenstein local rings associated to monomial curves, which supports Rossi’s question, saying that every Gorenstein local ring has a non-decreasing Hilbert function.  In the second part, we study strong indispensable minimal free resolutions of semigroup rings, focusing on the operation of the join of affine semigroups, which provide class of examples supporting Charalambous and Thoma's question on the class of lattice ideal which has a strong indispensable free resolution..
	\end{abstract}

	\maketitle

	\section{Introduction}

	Let $\mathbb{N}$ denote the set of non-negative integers and $\mathbb{K}$ 
	denote a field. Let 
	$e\geq 3$ and $\mathbf{\underline n} = (n_{1}, \ldots, n_{e})$ be a sequence of 
	$e$ distinct positive integers with $\gcd(\mathbf{\underline n})=1$. Let us assume 
	that the numbers $n_{1}, \ldots, n_{e}$ generate the numerical semigroup 
	$\Gamma(n_1,\ldots, n_e) = \langle n_{1}, \ldots , n_{e} \rangle = 
	\lbrace\sum_{j=1}^{e}z_{j}n_{j}\mid z_{j}\in \mathbb{N}\rbrace$ 
	minimally, that is, if $n_i=\sum_{j=1}^{e}z_{j}n_{j}$ for some non-negative 
	integers $z_{j}$, then $z_{j}=0$ for all $j\neq i$ and $z_{i}=1$. We often write 
	$\Gamma$ in place of $\Gamma(n_1,\ldots, n_e)$, when there is no confusion regarding 
	the defining sequence $n_{1}, \ldots, n_{e}$. Let 
	$\eta: R = \mathbb{K}[x_1,\,\ldots,\, x_e]\rightarrow \mathbb{K}[t]$ be the mapping defined by 
	$\eta(x_i)=t^{n_i},\,1\leq i\leq e$. The ideal $\ker (\eta) = \frak{p}(n_1,\ldots, n_e)$ 
	(or simply $\frak{p}(\Gamma)$) is called 
	the defining ideal of $\Gamma(n_1,\ldots, n_e)$ and it defines the 
	affine monomial curve 
	$\{(u^{n_{1}},\ldots , u^{n_{e}})\in \mathbb{A}^{e}_{\mathbb{K}}\mid u\in \mathbb{K}\} =: 
	C(n_{1},\ldots,n_{e})$ (or simply $C(\Gamma)$). We write 
	$\mathbb{K}[x_{1},\ldots,x_{e}]/\frak{p}(n_1,\ldots, n_e) =: \mathbb{K}[\Gamma(n_1,\ldots, n_e )]$ 
	(or simply $\mathbb{K}[\Gamma]$), which is called the semigroup ring for 
	the semigroup $\Gamma(n_1,\ldots, n_e )$. It is known that $\frak{p}(\Gamma)$ is 
	generated by the binomials $x^{a}-x^{b}$, where $a$ and $b$ are $e$-tuples of non-negative 
	integers with $\eta(x^{a})=\eta(x^{b})$. 
	
	The theory of numerical semigroups has been developed mainly in connection with the study of curve singularities. Numerical semigroup rings are one-dimensional domains and hence Cohen-Macaulay; these are 
	the coordinate rings of affine monomial curves. A 
	lot of interesting studies have been undertaken 
	by several authors from the viewpoints of singularities, homology and also from purely 
	semigroup theoretic aspects. Our main object in this paper is to explore the gluing of numerical semigroup rings and produce projective closure of a monomial curve and tangent cone of the monomial curve with the Cohen-Macaulay and Gorenstein properties. The concept of gluing was introduced by Rosales in \cite{Rosales} and used by several authors to produce new examples of numerical semigroup with nice property such as set-theoretic and ideal-theoretic complete intersection, Cohen-Macaulay, Gorenstein etc. In this paper, we particularly study the projective closure of gluing of numerical semigroup rings and then the tangent cone of gluing of numerical semigroup rings.
		
	\smallskip
	
	\noindent \emph{Projective Closure of Numerical Semigroup rings.}
		Let $n_{e}>n_{i}$ for all $i<e$, and $n_{0}=0$. We define the 
	semigroup $\overline{\Gamma(n_{1},\ldots n_{e})}  
	= \langle \{(n_{i},n_{e}-n_{i})\mid 0\leq i\leq e\} \rangle  = \lbrace \sum_{i=0}^{e}z_{i}(n_{i},n_{e}-n_{i}) \mid 
	z_{i}\in\mathbb{N}\rbrace$, often written as $\overline{\Gamma}$. 
	Let $\eta^{h}:S=\mathbb{K}[x_{0},\ldots,x_{e}]\longrightarrow \mathbb{K}[s,t]$ 
	be the $\mathbb{K}$-algebra map defined as $\eta^{h}(x_{i})=t^{n_{i}}s^{n_{e}-n_{i}}, 0\leq i \leq e$  and 
	$ \ker(\eta^{h}) = \overline{\mathfrak{p}(n_{1},\ldots , n_{e})}$ (or simply 
	$\overline{\mathfrak{p}(\Gamma)}$). The homogenization of 
	$\frak{p}(n_1,\ldots, n_e)$ with respect to the variable $x_{0}$ is 
	$\overline{\mathfrak{p}(n_{1},\ldots n_{e})}$, which defines 
	$\{[(v^{n_{e}}: v^{n_{e}-n_{1}}u^{n_{1}}: \cdots: u^{n_{e}})]\in\mathbb{P}^{e}_{\mathbb{K}}\mid u, v\in \mathbb{K}\} =: \overline{C(n_{1},\ldots,n_{e})}$ (or simply $\overline{C(\Gamma)}$),  
	the projective closure of the affine monomial curve $C(n_{1},\ldots,n_{e})$. 
	The $\mathbb{K}$-algebra 
	$\mathbb{K}[x_{0},\ldots,x_{e}]/\overline{\mathfrak{p}(n_{1},\ldots , n_{e})} =: 
	\mathbb{K}[\overline{\Gamma(n_1,\ldots, n_e )}]$ (or simply $\mathbb{K}[\overline{\Gamma}]$) 
	denotes the coordinate ring. It can be proved that 
	$\overline{C(n_{1},\ldots,n_{e})}$ is a projective curve, which is said to be arithmetically Cohen-Macaulay if the vanishing ideal $\overline{\mathfrak{p}(n_{1},\ldots , n_{e})}$ 
	is a Cohen-Macaulay ideal. 
	 The projective closure of a numerical semigroup ring is defined by an affine semigroup in $\mathbb{N}^2$. The Cohen-Macaulay property and many other properties, such as Betti numbers, Goresntein, etc., of 
	numerical semigroup rings are not preserved under the operation of projective closure. 
	It is not easy to find examples 
	where the projective closure retains the Cohen-Macaulay property; see 
	\cite{ACM}. Gluing of affine semigroup can be an effective technique for producing projective closure with the Cohen-Macaulay property, and from gluing perspective, Saha et al. in \cite{Saha-Gluing} asked the following question.
	
	\medspace
	
	\noindent\emph{Question 1}(\cite[Question 1]{Saha-Gluing}).\label{Question} Suppose the projective 
	closures of two affine monomial curves are arithmetically Cohen-Macaulay 
	(respectively Gorenstein). Which conditions on gluing of numerical semigroups 
	ensure that the 
	arithmetically Cohen-Macaulay 
	(respectively Gorenstein) properties are preserved for the projective closure? 
	
	An answer to the above Question would help us create 
a large family of affine monomial curves with arithmetically 
Cohen-Macaulay (respectively Gorenstein) projective closure. Saha et al. in \cite{Saha-Gluing} answer this question for \emph{star gluing} and \emph{simple gluing}. In particular, they proved that the Cohen-Macaulay (respectively Gorenstein) property preserves for the projective closure of star gluing of numerical semigroup rings provided projective closures of original numerical semigroups are Cohen-Macaulay (respectively Gorenstein). Feza Arslan et al. in \cite{Arslan} introduced the concept of 
\textit{nice gluing} to	give infinitely many families of $1$-dimensional local rings with Cohen 
Macaulay tangent cone. We prove that, the projective closure of a nice gluing of numerical semigroup rings is Cohen-Macaulay (respectively Gorenstein) under some conditions.

	\medspace
	
	\noindent\emph{Associated graded ring of numerical semigroup rings.}
	For a numerical semigroup $\Gamma$, $\mathbb{K}[\Gamma]$ is the coordinate ring of the affine monomial 
	curve $C(\Gamma)$. The tangent cone of $C(\Gamma)$ is the associated graded ring  $\mathrm{gr}_{\frak{m}}(\mathbb{K}[\Gamma])=\oplus_{i=0}^{\infty}\mathfrak{m}^{i}/\mathfrak{m}^{i+1}$, 
	with respect to the maximal ideal $\frak{m}=(t^{s} : s \in \Gamma\setminus \{0\})\mathbb{K}[\Gamma]$ 
	at the origin. The algebraic properties of $\mathrm{gr}_{\frak{m}}(\mathbb{K}[\Gamma])$ is a  
	focal point of recent research, with substantial progress happening in the directions of 
	finding its defining equations, Hilbert functions, Cohen-Macaulayness and Betti numbers. 	In this section, we study the Hilbert functions of local rings associated to affine
	monomial curves obtained by using the technique of gluing numerical semigroups. 
	Arslan et al. \cite{Arslan} give large
	families of local rings with non-decreasing Hilbert functions using nice gluing, which are, in fact, special types of
	gluings. 
		
	In this article we answer the following question for \emph{star gluing}: If $C(\Gamma_1)$ and $C(\Gamma_2)$ have Cohen-Macaulay
	tangent cones, is the tangent cone of the monomial curve $C(\Gamma)$ obtained by gluing
	these two monomial curves necessarily Cohen-Macaulay? The following example
	shows that the answer is no.

	\begin{example}[\cite{Arslan}, Example 2.2]
		Let $C(\Gamma_1)$ and $C(\Gamma_2)$ be the monomial curves $C(\Gamma_1) = C(5, 12)$ and $C(\Gamma_2)$ =
		$C(7, 8)$. Obviously, they have Cohen-Macaulay tangent cones. By a gluing of $\Gamma_1$ and $\Gamma_2$, we obtain the monomial curve $C = C(21 \times 5, 21 \times 12, 17 \times 7, 17 \times 8)$. $C$ has a non Cohen-Macaulay tangent cone.
	\end{example}
	
	We demonstrate that \emph{star gluing} is also an effective technique for constructing large families of $1$-dimensional Gorenstein local rings associated to monomial curves, which
	support the question due to Rossi saying that every Gorenstein local ring has
	non-decreasing Hilbert function \cite{Arslan-2}.
	
	\medspace
	
	\emph{Extension of affine semigroups.}
	In the numerical case, the finiteness of $\mathbb{N}\setminus \Gamma$ implies that there exists at least a positive integer $a \in \mathbb{N}\setminus \Gamma$ such that $a + \Gamma \setminus \{0\} \subseteq  \Gamma  \,(\text{provided that } \Gamma \neq \mathbb{N})$. These integers are called
	pseudo-Frobenius numbers and the biggest one is the so-called Frobenius number.  But for affine semigroups in $\mathbb{N}^d$, the existence 
	of such elements is not always guaranteed. The study of pseudo-Frobenius elements in affine semigroups over $\mathbb{N}^d$ is done in \cite{MPD}, where the authors consider the complement of the 
	affine semigroup in its rational polyhedral cone. In \cite{MPD}, the authors prove 
	that an affine semigroup $\Gamma$ has pseudo-Frobenius elements if and only if the length of the 
	graded minimal free resolution of the corresponding semigroup ring is maximal. Affine 
	semigroups having pseudo-Frobenius elements are called maximal projective dimension (MPD for short) semigroups. Projective closure which are non Cohen-Macaulay are comes under the class of maximal projective dimension semigroups. 
	Let $\Gamma$ be an MPD-semigroup and let $\mathrm{Cone}(\Gamma)$ denote the rational polyhedral 
	cone of $\Gamma$. Set $\mathcal{H}(\Gamma):=(\mathrm{Cone}(\Gamma) \setminus \Gamma)\cap \mathbb{N}^d$. For a fixed term order $\prec$ on $\mathbb{N}^d$, Bhardwaj et al. in \cite{Pseudo} introduced the $\prec$-symmetric semigroup, provided $\mathrm{F}(\Gamma)_{\prec}=\rm{max}_{\prec}\mathcal{H}(\Gamma)$ exists.
	
	One of the widely studied class of numerical semigroups is symmetric semigroups. The motivation to study 
	these semigroups comes from the work E. Kunz, who proved that a one-dimensional 
	analytically irreducible Noetherian local ring is Gorenstein if and only if its value semigroup is symmetric (See \cite{Kunz}). In other words, a numerical semigroup is symmetric if and only if 
	the associated semigroup ring is Gorenstein.  In \cite{Watanabe},  Watanabe gives a special
	technique of producing symmetric numerical semigroup. Later, Stamate in \cite{Betti} called this technique to \emph{simple gluing} and used this to produce Cohen-Macaulay tangent cone of numerical semigroup rings. In \cite[Theorem 4.11]{Saha-Gluing} Saha et al. proved that the projective closure of numerical semigroup ring obtained from the simple gluing is Cohen-Macaulay (respectively Gorenstein) provided projective closure of original numerical semigroup ring is Cohen-Macaulay (respectively Gorenstein). In this paper, we study the simple gluing of MPD semigroups and prove that simple gluing preserves the $\prec$-symmetric MPD semigroups provided original MPD semigroup should be $\prec$-symmetric.
	
	\smallskip
	
	\emph{Strong Indispensable Free Resolution.}
	Indispensable binomials are those that appear in every minimal binomial generating set
	up to a constant multiple. Strongly indispensable binomials are those appearing in every minimal
	generating set, up to a constant multiple. In the same vein, as introduced for the first time by
	Charalambous and Thoma in \cite{SIFR}, strongly indispensable higher syzygies are those appearing
	in every minimal free resolution. Semigroups all of whose higher syzygy modules are generated minimally by strongly indispensable elements are said to have a strongly indispensable minimal
	free resolution, SIFRE for short. The statistical models having SIFREs or equivalently having
	uniquely generated higher syzygy modules are a subclass of those having a unique Markov basis
	and therefore have a better potential statistical behaviour. 
	Motivated by the third question stated by Charalambous and Thoma at the end of \cite{SIFR}, our main
	aim in this article is to identify some semigroups having SIFREs.
	It is difficult to construct examples having SIFREs.
	\c{S}ahin and Stella in \cite{Sahin-SIFR} produces SIFR using the gluing of affine semigroups. Saha et al. in \cite{Join}, introduced the notion of  \emph{Join of affine semigroups} which has some nice property. In Section 3, we prove that $\Gamma_1$ and $\Gamma_2$ have SIFRs if and only if their join is also have SIFR (See Theorem \ref{Join-SIFR} )
	
	\section{Gluing of monomial curves}
	
	We start this section with the example.
	
	\begin{example}{\rm 
			We consider monomial curves $\overline{C(3,5)}$ and  $\overline{C(7,12)}$. 
			They are Cohen-Macaulay (Gorenstein), but their gluing with respect to the elements 
			$p=8, q=19$ is $\overline{C(57,95,56,96)}$, which is not even arithmetically 
			Cohen-Macaulay (verified with the help of MACAULAY 2 \cite{Macaulay}).
		}
	\end{example}
	
	\noindent This motivates authors in \cite{Saha-Gluing} to ask Question 1, described in the Introduction. 
	We summarise two results which will be useful for our purpose. 
	
	\begin{lemma}[\cite{CMC}, Lemma 2.1]\label{Cri-Grob}
		Let $I$ be an ideal in $R=\mathbb{K}[x_{1},\ldots,x_{e}]$ and 
		$I^{h}\subset R[x_{0}]$ its homogenization, 
		with respect to the variable $x_{0}$. Let $<$ be any reverse lexicographic monomial order on $R$ and $<_{0}$ 
		the reverse lexicographic monomial order on $R[x_{0}]$, extended from $R$, such that $x_{i}>x_{0}$. 
		If $\{f_{1},\ldots,f_{d}\}$ is the reduced Gr\"{o}bner basis for $I$, 
		with respect to $<$, then $\{f_{1}^{h},\ldots,f_{d}^{h}\}$ is the reduced Gr\"{o}bner basis 
		for $I^{h}$, with respect to $<_{0}$, and $\mathrm{in}_{<_{0}}(I^{h})=(\mathrm{in}_{<}(I))R[x_{0}]$.
		
	\end{lemma} 	
	
	The other one that is useful for checking the arithmetically Cohen-Macaulay property of the projective closure. It states the following: 
	
	\begin{theorem}[\cite{CMC}, Theorem 2.2]\label{Cond-CM}
		Let $\mathbf{n} = (n_{1},\ldots,n_{e})$ be a sequence of positive integers with $n_{e}> n_{i}$ for all $i<n$. Let $<$ 
	be any reverse lexicographic order on $R=\mathbb{K}[x_{1},\ldots,x_{e}]$, such that $x_{i}>x_{e}$, for all $1\leq i<e$. 
	Let $<_{0}$ be the induced reverse lexicographic order on $R[x_{0}]$, where $x_{e}>x_{0}$. Then the following 
	conditions are equivalent:
	\begin{enumerate}
		\item[(i)] The projective monomial curve $\overline{C(n_{1},\ldots,n_{e})}$ is arithmetically Cohen-Macaulay.
		\item[(ii)] $\mathrm{in}_{<_{0}}((\mathfrak{p}(n_{1},\ldots,n_{e}))^{h})$ (homogenization w.r.t. $x_{0}$) is a Cohen-Macaulay ideal.
		\item[(iii)] $\mathrm{in}_{<}(\mathfrak{p}(n_{1},\ldots,n_{e}))$ is a Cohen-Macaulay ideal.
		\item[(iv)] $x_{e}$ does not divide any element of $G(\mathrm{in}_{<}(\mathfrak{p}(n_{1},\ldots,n_{e})))$.
	\end{enumerate}

	\end{theorem}
	
	\smallskip

	\begin{definition}[\cite{Rosales}] 
		{\rm 
			Let $\Gamma_{1} = \Gamma(m_{1},\dots, m_{l})$ and 
			$\Gamma_{2} = \Gamma(n_{1},\dots, n_{k})$ be two numerical semigroups, 
			with $m_{1} < \cdots < m_{l}$ and $n_{1} < \cdots < n_{k}$. 
			Let $p =b_{1}m_{1} +\cdots +b_{l}m_{l} \in \Gamma_{1}$ and 
			$q = a_{1}n_{1} +\cdots +a_{k}n_{k} \in\Gamma_{2}$ be two positive 
			integers satisfying $\gcd(p, q) = 1 $, with 
			$p \notin \{m_{1}, \dots ,m_{l} \}$, $q \notin \{n_{1}, \ldots ,n_{k}\}$ 
			and $\{qm_{1},\ldots , qm_{l}\}\cap \{pn_{1},\ldots , pn_{k}\} = \emptyset$. 
			The numerical semigroup $\Gamma_{1}\#_{p,q} \Gamma_{2}=\langle qm_{1},\ldots , qm_{l}, pn_{1},\ldots, pn_{k}\rangle$ is called a \textit{gluing} of the semigroups $\Gamma_{1}$ and $\Gamma_{2}$ with 
			respect to $p$ and $q$.
		}
	\end{definition}
	
	\begin{remark}[\cite{SMR}, Lemma 2.2] 
		If $\Gamma$ be obtained by gluing 
	$\Gamma_{1} = \Gamma(m_{1},\dots, m_{l})$ and 
	$\Gamma_{2} = \Gamma(n_{1},\dots, n_{k})$ 
	with respect to $p = \sum_{i=1}^{l}b_{i}m_{i}$ and $q = \sum_{i=1}^{k}a_{i}n_{i}$,  
	and if the defining ideals $\mathfrak{p}(\Gamma_{1}) \subset \mathbb{K}[x_{1},\ldots, x_{l}]$ and 
	$\mathfrak{p}(\Gamma_{2}) \subset \mathbb{K}[y_{1},\ldots, y_{k}]$ are generated by the sets 
	$G_{1} = \{f_{1},\ldots, f_{d}\} $ and $G_{2} = \{g_{1},\ldots, g_{r}\}$ respectively, 
	then the defining ideal 
	$\frak{p}(\Gamma)\subset R = \mathbb{K}[x_{1},\ldots, x_{l}, y_{1},\ldots, y_{k}]$ is generated by the set 
	$G = G_{1}\cup G_{2}\cup \{\rho\}$, where 
	$\rho=x_{1}^{b_{1}}\dots x_{l}^{b_{l}}-y_{1}^{a_{1}}\dots y_{k}^{a_{k}}$.
	
	\end{remark}

	\begin{definition}[\cite{Arslan}, Definiton 2.3]\label{Nice-Gluing}{\rm 
			The numerical semigroup $\Gamma_{1}\#_{p,q} \Gamma_{2}$,  
			obtained by gluing of  $\Gamma_{1} = \Gamma(m_{1},\dots, m_{l})$ and $\Gamma_{2} = \Gamma(n_{1},\dots, n_{k})$, with respect to the positive integers $p$ and $q$, is 
			called a \textit{nice gluing} if $p = b_{1}m_{1}+\dots+b_l m_l \in \Gamma_{1}$ and $q = a_{1}n_{1}\in \Gamma_{2}$, 
			with $ b_{1}+b_{2}+\dots+b_{l} \geq a_{1}$.  If we take $q=a_1n_1+\dots+a_k n_k$ with $ b_{1}+b_{2}+\dots+b_{l} > a_{1}+\dots+a_k$. Then the numerical semigroup minimally generated by $\langle qm_1,\dots,qm_l,pn_1,\dots,pn_k\rangle $ is called generalized nice gluing of $\Gamma_1$ and $\Gamma_2$ with respect to $p,q$. We use the notation 
			$\Gamma_{1}\#^n_{p,q} \Gamma_{2}$ to denote the generalized nice gluing 
			of $\Gamma_{1}$ and $\Gamma_{2}$, with respect to the positive integers 
			$p$ and $q$.
			
		}
	\end{definition}

		\begin{remark} For the Cohen-Macaulayness of projective closure of the numerical semigroup $\Gamma=\Gamma_{1}\#^n_{p,q} \Gamma_{2}$, obtained by nice gluing of $\Gamma_{1}$ and $\Gamma_{2}$ with respect to $p,q$, it is important to find the largest integer from the generator 
			of the numerical semigroup $\Gamma_{1}\#^n_{p,q} \Gamma_{2}$. Choices for the largest integer $\Gamma$ is either $qm_l$ or $pn_k$. If the largest integer is $qm_l$ then $\overline{C(\Gamma)}$ need not be Cohen-Macaulay (See Example \ref{Counter}).
		\end{remark}
		
		\begin{example}\label{Counter}{\rm
			Let $\Gamma_1=\langle 5,7,11 \rangle$ and $\Gamma_2=\langle 25,28 \rangle$. Let $p=2\cdot 5+7=17,\, q=2\cdot 25=50$. Then $\Gamma=\Gamma_{1}\#^n_{17,50} \Gamma_{2}=\langle 250,350,550,425,476 \rangle$ and $\overline{C(\Gamma)}$ is not arithmetically Cohen-Macaulay (verified with the help of MACAULAY 2 \cite{Macaulay}).
		}
		\end{example}
		
		\noindent \textbf{Condition A}: Let $\mathrm{lcm}(b_i,\alpha_i)\neq b_i$  for all $1 \leq i \leq l$, where  $x_{1}^{\alpha_1}\dots x_{l}^{\alpha_l}=\mathrm{LM}(f)$ and $f $ is an arbitrary element of a reduced Gr\"{o}bner basis of $\mathfrak{p}(\Gamma_1)$ and $b_i$ is according to Definition \ref{Nice-Gluing}.   We deonte $\mathrm{lcm}(m,n)$ by $[m,n]$.
		This condition has been defined in such a way that $S$-polynomial of $\rho$ and $f$ reduces to zero after a appropriate multidivison by $G$ (defined in Lemma \ref{GG}).
 		
		\begin{lemma}\label{GG}
			Let $\Gamma_{1} = \Gamma(m_{1},\dots, m_{l})$ and $\Gamma_{2} = \Gamma(n_{1},\dots, n_{k})$ be two numerical semigroups with $m_{1} < \cdots < m_{l}$ and $n_{1} < \cdots < n_{k}$.
			Let 
			$p = b_{1}m_{1}+\dots+p_l m_l \in \Gamma_{1}$ and 
			$q = a_{1}n_{1} +\cdots +a_{k}n_{k} \in\Gamma_{2}$ be 
			two positive integers satisfying 
			$\gcd(p, q) = 1 $, such that $\Gamma=\Gamma_{1}\#^n_{p,q} \Gamma_{2}$ is the generalized nice gluing of $\Gamma_{1}$ 
			and $\Gamma_{2}$ with $b_1,\dots,b_l$ satisfy Condition $A$. Suppose that $G_{1}$ and $G_{2}$ 
			are Gr\"{o}bner bases of $\frak{p}(\Gamma_{1})$ and 
			$\frak{p}(\Gamma_{2})$, with respect to the degree 
			reverse lexicographic ordering induced by 
			$x_{1}>\dots>x_{l}$ and $y_{1}>\dots>y_{k}$ 
			respectively. Then 
			$$G_{1}^{h} \cup G_{2}^{h} \cup \{x^{b_{1}}_{1}\dots x_{l}^{b_l}-x_{0}^{b_{1}\dots + b_l-(a_{1}+\cdots+a_{k})}y^{a_{1}}_{1}y^{a_{2}}_{2}\dots y^{a_{k}}_{k}\}$$ is a Gr\"{o}bner bases of 
			$\overline{\frak{p}(\Gamma)}$  with respect to the degree 
			reverse lexicographic ordering induced by
			$x_{1} > \dots>x_{l} > y_{1} > \dots > y_{k} > x_{0}$ or $y_{1} > \dots>y_{k} > x_{1} > \dots > x_{l} > x_{0}$, where $G_{1}^{h}$, $G_{2}^{h}$ denotes  
			homogenization of $G_1$ and $G_2$ with respect to $x_{0}$. 
		\end{lemma}
		
		\proof Let $G_{1} = \{f_{1},\ldots, f_{d}\}$ be a reduced Gr\"{o}bner basis of the ideal $\mathfrak{p}(\Gamma_{1}) \subset \mathbb{K}[x_{1},\ldots, x_{l}]$, 
		with respect to the degree
		reverse lexicographical ordering induced by 
		$x_{1}> \cdots> x_{l}$ and $ G_{2} = \{g_{1}, \ldots , g_{r}\}$ be
		a reduced Gr\"{o}bner basis of the ideal 
		$\mathfrak{p}(\Gamma_{2}) \subset \mathbb{K}[y_{1}, \ldots , y_{k}]$, with respect to the degree
		reverse lexicographical ordering induced by $y_{1}>\ldots > y_{k}$.
		We show that $G=\{f_{1}, \dots, f_{d}, g_{1}, \dots , g_{r}, 
		\rho=x_{1}^{b_{1}}\dots x_{l}^{b_l}-y^{a_{1}}_{1}y^{a_{2}}_{2} \dots y^{a_{k}}_{k} \}$ 
		is a Gr\"{o}bner basis  of $\frak{p}(\Gamma)$, with respect to the degree
		reverse lexicographical ordering induced by $x_{1} >\cdots>x_{l}>y_{1}>\cdots>y_{k}$. 
		By Lemma 2.2 in \cite{SMR}, 
		the defining ideal $\mathfrak{p}(\Gamma)$ of the affine curve $C(\Gamma)$, obtained by gluing, 
		is generated by the set $G$. 
		Let $S(f, g)$ denote the $S$-polynomial of polynomials $f$ and $g$ in 
		$\mathbb{K}[x_{1},\ldots,x_{l},y_{1},\ldots,y_{k}]$. 
		With respect to the said monomial order, the leading monomial of any element of $G_{1}$ is the product of 
		monomials of the form $x_{i}^{\alpha_{i}}$ only, 
		and $G_{2}$ is the product of monomials of the form $y_{j}^{\beta_{j}}$ only, for some 
		non-negative integers $\alpha_{i}$, $\beta_{j}$ and $1\leq i \leq l$, $1\leq j \leq k$. 
		Therefore $S(f_{i},g_{j}) \rightarrow_{G} 0$. It is clear from the condition of generalized nice gluing that 
		$\mathrm{LM}(x_{1}^{b_{1}}\dots x_{l}^{b_{l}}-y^{a_{1}}_{1}y^{a_{2}}_{2}\dots y^{a_{k}}_{k})=x_{1}^{b_{1}}\dots x_{l}^{b_{l}}$. 
		Let  $f_i=x_1^{\alpha_1}\dots x_{l}^{\alpha_{l}}-x_1^{\beta_1}\dots x_{l}^{\beta_{l}}$ be an arbitrary element of $G_1$ with $\mathrm{LM}(f_i)=x_1^{\alpha_1}\dots x_{l}^{\alpha_{l}}$.
		Now $S(\rho,f_i)=-x_1^{[b_1,\alpha_1]-b_1}\dots x_{l}^{[b_{l},\alpha_{l}]-b_{l}}(y_1^{a_1}\dots y_{k}^{a_k})+x_1^{[b_1,\alpha_1]-\alpha_1}\dots x_{l}^{[b_{l},\alpha_{l}]-\alpha_{l}}x_1^{\beta_1}\dots x_l^{\beta_{l}}$. Since $\sum_{s=1}^{l}([b_s,\alpha_s]-b_s)+a_1+\dots +a_k  \geq 	\sum_{s=1}^{l}([b_s,\alpha_s]-\alpha_s)+\beta_1+\dots +\beta_k$, we have $\mathrm{LM}(S(\rho,f_i))=	x_1^{[b_1,\alpha_1]-b_1}\dots x_{l}^{[b_{l},\alpha_{l}]-b_l}(y_1^{a_1}\dots y_{k}^{a_k})$	
		and using multi variable division algorithm $S(\rho,f_i)$ reduced to zero after division by $G_1 \cup G_2 \cup \{\rho\}$ due to condition A.
		Since for any $g_i \in G_2$, $\mathrm{gcd}(\mathrm{LM}(\rho)=x_{1}^{b_{1}}\dots x_{l}^{b_{l}},\mathrm{LM}(g_i)= \text{some monomial in } y_j's)=1$,  we have $S(\rho,g_i)$ reduces to zero after division by $G_1 \cup G_2 \cup \{\rho\}$.
		By the Buchberger's criterion, $G=G_1 \cup G_2 \cup \{\rho\}$ is a Gr\"{o}bner  basis 
		of $\frak{p}(\Gamma)$ and by Lemma \ref{Cri-Grob} , $G^{h}=G_{1}^{h} \cup G_{2}^{h} \cup \{x^{b_{l}}_{l}-x_{0}^{b_{l}-(a_{1}+\cdots+a_{k})}y^{a_{1}}_{1}y^{a_{2}}_{2}\dots y^{a_{k}}_{k}\}$ 
		is a Gr\"{o}bner basis of $\overline{\mathfrak{p}(\Gamma)}$. \qed
		
		\medspace
		
		\begin{theorem}\label{GT} Assuming the hypothesis of Theorem \ref{GG}, if the associated projective closure $\overline{C(\Gamma_{1})}$ and $\overline{C(\Gamma_{2})}$ are arithmetically Cohen-Macaulay, then 
			\begin{itemize}
				\item If $pn_k$ is the largest integer among the generators of $\Gamma=\Gamma_{1}\#^n_{p,q} \Gamma_{2}$ then the projective closure $\overline{C(\Gamma)}$ associated to $\Gamma$ is arithmetically Cohen-Macaulay.
				
				\item If $qm_l$ is the largest integer among the generators of $\Gamma=\Gamma_{1}\#^n_{p,q} \Gamma_{2}$ then the projective closure $\overline{C(\Gamma)}$ associated to $\Gamma$ is never arithmetically Cohen-Macaulay.
			\end{itemize}
		\end{theorem}
		
		\begin{proof}
			
			\begin{enumerate}
				\item Assume that $pn_k$ is the largest integer among the generators of $\Gamma=\Gamma_{1}\#^n_{p,q} \Gamma_{2}$ then $y_k$ will be the deciding variable for the Cohen-Macaulayness of $\overline{C(\Gamma)}$. From Lemma \ref{GG}, $G^{h}=G_{1}^{h} \cup G_{2}^{h} \cup \{x^{b_{l}}_{l}-x_{0}^{b_{l}-(a_{1}+\cdots+a_{k})}y^{a_{1}}_{1}y^{a_{2}}_{2}\dots y^{a_{k}}_{k}\}$ is a Gr\"{o}bner bases of $\overline{\frak{p}(\Gamma)}$, 
				with respect to the degree reverse lexicographic ordering $x_{1}>\dots>x_{l}>y_{1}>\dots>y_{k}>x_{0}$, 
				where $G_{1}$ and $G_{2}$ are Gr\"{o}bner bases of $\frak{p}(\Gamma_{1})$ and $\frak{p}(\Gamma_{2})$, 
				with respect to the degree reverse lexicographic ordering $x_{1}>\dots>x_{l}$ and $y_{1}>\dots>y_{k}$ 
				respectively. Since $ \overline{C(\Gamma_{2})}$ is  arithmetically 
				Cohen-Macaulay, by Theorem \ref{Cond-CM}, 
				$y_{k}$ does not divide the leading monomial of any element in $G_{1}^{h}$ and $G_{2}^{h}$, and also $\mathrm{LM}(x^{b_{1}}_{1}\dots x_{l}^{b_l}-x_{0}^{b_{l}-(a_{1}+\cdots+a_{k})}y^{a_{1}}_{1}y^{a_{2}}_{2}\dots y^{a_{k}}_{k})=x^{b_{1}}_{1}\dots x_{l}^{b_l}$. Thus, $y_{k}$ does not divide the leading monomial of any element in $G^{h}$, 
				which is a Gr\"{o}bner basis with respect to to the degree reverse lexicographic ordering 
				induced by $x_{1} >...>x_{l}>y_{1}>...>y_{k}>x_{0}$. Thus, by Theorem \ref{Cond-CM}, 
				$\overline{C(\Gamma)}$ is arithmetically Cohen-Macaulay.
				
				\item If $qm_l$ is the largest integer among the generators of $\Gamma=\Gamma_{1}\#^n_{p,q} \Gamma_{2}$ then $x_l$ is the deciding variable and $x_l$ divides $\mathrm{LM}(\rho)$, therefore  by Theorem \ref{Cond-CM}, $\overline{C(\Gamma)}$ is not arithmetically Cohen-Macaulay.
			\end{enumerate}	
		 
	\end{proof}
		
		\begin{corollary}\label{non-CM}
			Under the hypothesis of generalized nice gluing of $\Gamma_{1}$ and $\Gamma_{2}$ with respect to $p,q$, 
			let $\Gamma=\Gamma_{1} \#^{n}_{p,q} \Gamma_{2}$. If $\overline{C(\Gamma_1)}$ and $\overline{C(\Gamma_2)}$ are  Gorenstein then $\overline{C(\Gamma)}$ is Gorenstein.
		\end{corollary}
		\begin{proof}
			Proof is directly follows from Theorem \ref{GG}, \cite[Proposition 3.3]{SMR} and \cite[Theorem 3.13]{Saha-Gluing}.
		\end{proof}
		
		\medspace
		
		\begin{example}
			{\rm
				Let $\Gamma_1=\langle 3,5,7 \rangle$ and $\Gamma_2=\langle 9,11 \rangle $ with $p=3\cdot3+5=14$, $q=9\cdot2+11=29$. Then $\Gamma=\Gamma_{1}\#^n_{p,q} \Gamma_{2}=\langle 87,145,203,126,154 \rangle$. 	
				Note that $x_3$ is the variable corresponding to the largest integer of the minimal generator of $\Gamma$.  Thus $\overline{C(\Gamma)}$ is not arithmetically Cohen-Macaulay (verified with the help of MACAULAY 2 \cite{Macaulay}). But take $p=2\cdot 3+3\cdot 5=21, q=29$, then $\Gamma=\Gamma_{1}\#^n_{p,q} \Gamma_{2}=\langle 87,145,203,189,231 \rangle$. Here $y_2$ is the variable corresponding to the largest integer of the minimal generator of $\Gamma$.  Thus $\overline{C(\Gamma)}$ is  arithmetically Cohen-Macaulay
			}
		\end{example}
				
				\medspace

		\subsection{Gluing of tangent Cone}
		
		We recall that an affine monomial curve $C(\Gamma)$ is a curve with generic zero $(t^{n_1} ,\dots,t^{n_e} )$ in the affine $e$-space $\mathbb{A}^e$ over an algebraically closed field $\mathbb{K}$, where $\Gamma=\langle n_1,\dots,n_e \rangle$. The semigroup ring associated
		to the monomial curve $C(\Gamma)$ is $\mathbb{K}[[t^{n_1} ,\dots,t^{n_e} ]])$, and the Hilbert
		function of this local ring is the Hilbert function of its associated graded ring
		$\mathrm{gr}_{\mathfrak{m}}(\mathbb{K}[[t^{n_1} ,\dots,t^{n_e} ]])$, which is isomorphic to the ring $\mathbb{K}[x_1,\dots,x_e]/\mathfrak{p}^{\star}(\Gamma)$, where
		$\mathfrak{p}(\Gamma)$ is the defining ideal of $C(\Gamma)$ and $\mathfrak{p}^{\star}(\Gamma)$ is the ideal generated by the polynomials	$f^{\star}$, with $f$ in $\mathfrak{p}(\Gamma)$ and $f^{\star}$ being the homogeneous summand of $f$ of least degree.
		In other words, $\mathfrak{p}^{\star}(\Gamma)$ is the defining ideal of the tangent cone of $C(\Gamma)$ at $0$.

		\medspace

		Let us recall the definition of \emph{star gluing} of numerical semigroups.
		
		\begin{definition}\label{Star-Gluing}{\rm 
				The numerical semigroup $\Gamma_{1}\#_{p,q} \Gamma_{2}$,  
				obtained by gluing of  $\Gamma_{1} = \Gamma(m_{1},\dots, m_{l})$ and $\Gamma_{2} = \Gamma(n_{1},\dots, n_{k})$, with respect to the positive integers $p$ and $q$, is 
				called a \textit{generalized star gluing} if $p = b_{1}m_{1}+\dots +b_l m_l \in \Gamma_{1}$ and $q = a_{1}n_{1}+a_{2}n_{2}+\dots+a_{k}n_{k} \in \Gamma_{2}$, 
				with $ a_{1}+a_{2}+\dots+a_{k} < b_{l}+\dots +b_l.$ We use the notation 
				$\Gamma_{1}\star_{p,q} \Gamma_{2}$ to denote the generalized star gluing 
				of $\Gamma_{1}$ and $\Gamma_{2}$, with respect to the positive integers 
				$p$ and $q$.
			}
		\end{definition}
		
		To prove the Cohen-Macaulayness of tangent cone of monomial curve obtained by star gluing, we first write the criterion for checking
		the Cohen-Macaulayness of the tangent cone of a monomial curve
		
		\begin{lemma}\label{Criterion}
			Let $\Gamma=\langle n_1,\dots,n_k\rangle$ be a numerical semigroup minimally generated by $n_1 < \dots <n_k$, let $C=C(\Gamma)$ be the associated  monomial curve and let $G=\{f_1,\dots,f_s\}$ be a minimal Gr\"{o}bner basis of the ideal $\mathfrak{p}(\Gamma) \subset \mathbb{K}[x_1,\dots,x_k]$ with respect to the negative degree lexicographic ordering that makes $x_1$ the lowest variable. $C$ has Cohen-Macaulay tangent cone at the origin if and only $x_1$ does not divide $\mathrm{LM}(f_i)$ for $1 \leq i \leq s$.
		\end{lemma}

		\begin{proof}
			See \cite[Lemma 2.7]{Arslan}.
		\end{proof}

		\begin{remark}
		 For the Cohen-Macaulayness of tangent cone of affine monomial curve, it is important to find the smallest integer among the generator of associated numerical semigroup. In case of numerical semigroup $\Gamma_{1} \star_{p,q} \Gamma_{2}$, the possibility of smallest generator is either $qm_1$ or $pn_1$.
		\end{remark}
		
		\medspace
		
			\noindent \textbf{Condition B}: Let $\mathrm{lcm}(a_i,\alpha_i)\neq a_i$  for all $1 \leq i \leq k$, where  $y_{1}^{\alpha_1}\dots y_{k}^{\alpha_k}=\mathrm{LM}(g)$ and $g $ is an arbitrary element of a reduced Gr\"{o}bner basis of $\mathfrak{p}(\Gamma_2)$ and $a_i$ is according to Definition \ref{Star-Gluing} . This condition has been defined in such a way that $S$-polynomial of $\rho$ and $g$ reduces to zero after a appropriate multidivison by $G$ (defined in Theorem \ref{Tangent Cone}).
		
		\begin{theorem}\label{Tangent Cone}
			
			Let $\Gamma_{1} = \Gamma(m_{1},\dots, m_{l})$ and $\Gamma_{2} = \Gamma(n_{1},\dots, n_{k})$ be two numerical semigroups with $m_{1} < \cdots < m_{l}$ and $n_{1} < \cdots < n_{k}$.
			Let 
			$p = b_{1}m_{1}+\dots+p_l m_l \in \Gamma_{1}$ and 
			$q = a_{1}n_{1} +\cdots +a_{k}n_{k} \in\Gamma_{2}$ be 
			two positive integers satisfying 
			$\gcd(p, q) = 1 $, such that $\Gamma_{1}\star_{p,q} \Gamma_{2}$ is the generalized star gluing of $\Gamma_{1}$ 
			and $\Gamma_{2}$ with $b_1,\dots,b_l$ satisfy condition $B$.	If the associated monomial curves $C(\Gamma_1)$ and $C(\Gamma_2)$ have Cohen-Macaulay tangent cone at the origin.
				
			\begin{itemize}
				\item If $qm_1$ is the smallest integer among the generators of $\Gamma=\Gamma_{1}\star_{p,q} \Gamma_{2}$ then $C=C(\Gamma)$ also has a Cohen-Macaulay tangent cone at the origin.
				
				\item If $pn_1$ is the smallest integer among the generators of $\Gamma=\Gamma_{1}\star_{p,q} \Gamma_{2}$ then $C=C(\Gamma)$ has never a Cohen-Macaulay tangent cone at the origin.
			\end{itemize}			
	\end{theorem}		
	
		\begin{proof}
		Let $G_{1} = \{f_{1},\ldots, f_{d}\}$ be a reduced Gr\"{o}bner basis of the ideal $\mathfrak{p}(\Gamma_{1}) \subset \mathbb{K}[x_{1},\ldots, x_{l}]$, 
		with respect to the negative degree
		reverse lexicographical ordering induced by 
		$x_{l}> \cdots> x_{1}$ and $ G_{2} = \{g_{1}, \ldots , g_{r}\}$ be
		a reduced Gr\"{o}bner basis of the ideal 
		$\mathfrak{p}(\Gamma_{2}) \subset \mathbb{K}[y_{1}, \ldots , y_{k}]$, with respect to the negative degree
		reverse lexicographical ordering induced by $y_{l}>\ldots > y_{1}$.
		We show that $G=\{f_{1}, \dots, f_{d}, g_{1}, \dots , g_{r}, 
		\rho=x_{1}^{b_{1}}\dots x_{l}^{b_l}-y^{a_{1}}_{1}y^{a_{2}}_{2} \dots y^{a_{k}}_{k} \}$ 
		is a Gr\"{o}bner basis  of $\frak{p}(\Gamma)$, with respect to the negative degree
		reverse lexicographical ordering induced by $x_{l} >\cdots>x_{1}>y_{k}>\cdots>y_{1}$. 
		Let $S(f, g)$ denote the $S$-polynomial of polynomials $f$ and $g$ in 
		$\mathbb{K}[x_{1},\ldots,x_{l},y_{1},\ldots,y_{k}]$. 
		With respect to the said monomial order, the leading monomial of any element of $G_{1}$ is the product of 
		monomials of the form $x_{i}^{\alpha_{i}}$ only, 
		and $G_{2}$ is the product of monomials of the form $y_{j}^{\beta_{j}}$ only, for some 
		non-negative integers $\alpha_{i}$, $\beta_{j}$ and $1\leq i \leq l$, $1\leq j \leq k$. 
		Therefore $S(f_{i},g_{j}) \rightarrow_{G} 0$. It is clear from the condition of generalized star gluing and definition of negative degree reverse lexicographic ordering that 
		$\mathrm{LM}(x_{1}^{b_{1}}\dots x_{l}^{b_{l}}-y^{a_{1}}_{1}y^{a_{2}}_{2}\dots y^{a_{k}}_{k})=y_{1}^{a_{1}}\dots y_{k}^{a_{k}}$. 
		Let  $g_i=y_1^{\alpha_1}\dots x_{l}^{\alpha_{l}}-y_1^{\beta_1}\dots y_{l}^{\beta_{l}}$ be an arbitrary element of $G_2$ with $\mathrm{LM}(g_i)=y_1^{\alpha_1}\dots y_{l}^{\alpha_{l}}$.
		Now the spolynomial of $\rho$ and $g_i$ is given by $S(\rho,g_i)=-x_{1}^{b_1}x_{2}^{b_2}\dots x_{l}^{b_l}y_{1}^{[\alpha_1,a_1]-a_1}\dots y_k^{[\alpha_k,a_k]-a_k}-y_1^{[\alpha_1,a_1]-\alpha_1+\beta_1}\dots y^{[\alpha_k,a_k]-\alpha_k+\beta_k}$. Since $\sum_{s=1}^{k}([a_s,\alpha_s]-a_s)+b_1+\dots +b_l  >	\sum_{s=1}^{k}([a_s,\alpha_s]-\alpha_s)+\beta_1+\dots +\beta_k$, we have $\mathrm{LM}(S(\rho,f_i))=	-x_{1}^{b_1}x_{2}^{b_2}\dots x_{l}^{b_l}y_{1}^{[\alpha_1,a_1]-a_1}y_k^{[\alpha_k,a_k]-a_k}$,
		and using multi variable division algorithm $S(\rho,f_i)$ reduced to zero after division by $G_1 \cup G_2 \cup \{\rho\}$ due to condition B.
		Since for any $g_i \in G_2$, $\mathrm{gcd}(\mathrm{LM}(\rho)=x_{1}^{b_{1}}\dots x_{l}^{b_{l}},\mathrm{LM}(g_i)= \text{some monomial in } y_j's)=1$,  we have $S(\rho,g_i)$ reduces to zero after division by $G_1 \cup G_2 \cup \{\rho\}$.
		By the Buchberger's criterion, $G=G_1 \cup G_2 \cup \{\rho\}$ is a minimal Gr\"{o}bner  basis 
		of $\frak{p}(\Gamma)$.	
		
		If $qm_1$ is the smallest integer among the generators of $\Gamma$ and it is straightforward to check that $y_1$ do not divide the leading monomial of any elements of $G$ as $\Gamma_2$ has Cohen-Macaulay tangent cone, hence by Lemma \ref{Criterion},   	$C(\Gamma)$ has Cohen-Macaulay tangent cone at origin. If $pn_1$ is the smallest integer then we can see that $x_1$ divide $\mathrm{LM}(\rho)$, therefore again by Lemma \ref{Criterion}, $C=C(\Gamma)$ has never a Cohen-Macaulay tangent cone at the origin.	
		\end{proof}
		
		\smallskip
		
		\begin{remark}
			If we take two Gorenstein monomial curves with Cohen-Macaulay tangent cone then the monomial curve obtained by star gluing is also Gorenstein (See \cite{SMR}), and has a Cohen-Macaulay tangent cone.
			Since monomial curves with Cohen-Macaulay tangent cone has a non-decreasing Hilbert function,  hence we can produce infinitely many $1$-dimensional Gorenstein monomial curve with non-decreasing Hilbert function, which gives  affirmative answer to Rossi's question.
		\end{remark}		
		
		\smallskip
		
				\begin{example}{\rm
				Let $\Gamma_1=\langle 3,5,7 \rangle$ and $\Gamma_2=\langle 9,11 \rangle $ with $p=4.7=28$, $q=2.9+11=29$. Then $\Gamma=\Gamma_{1}\star_{p,q} \Gamma_{2}=\langle 87,145,203,189,231 \rangle$. Then $G=\{x_2^2-x_1 x_3,x_3^2-x_1^3x_2,x_2 x_3-x_1^4,y_2^9-y_1^{11},y_1^2 y_2-x_3^4\}$ is a reduced Gr\"{o}bner basis with respect to the negative degree reverse lexicographic ordering induced by $y_2>y_1>x_3>x_2>x_1$ of $\mathfrak{p}(\Gamma)$. 	
				Note that $x_1$ is the variable corresponding to the smallest integer of the minimal generator of $\Gamma$ and $x_1$ does not divide leading monomial of any element of $G$. Thus the tangent cone of $C(\Gamma)$ is Cohen-Macaulay.}
		\end{example}
			
		\medspace
		
		\subsection{Extension of MPD semigroup}\label{MPD}
		
		Let $\mathcal{A}=\{\mathbf{a}_1,\dots,\mathbf{a}_n\} \subset \mathbb{N}^d$ and let $\Gamma$ is a submonoid of $\mathbb{N}^d$ generated by $\mathcal{A}$. An affine semigroup $\Gamma$ is said to be maximal projective dimension semigroup (MPD-semigorup)
		if $\mathrm{pd}_R(\mathbb{K}[\Gamma])=n-1$. Equivalently, $\mathrm{depth}_{R}(\mathbb{K}[\Gamma])=1$. Consider the cone of $\Gamma$ in $\mathbb{N}^d$,
		\[
		\mathrm{Cone}(\Gamma):=\Big\{\sum_{i=1}^n \lambda_i \mathbf{a_i} : \lambda_i \in \mathbb{Q}_{\geq 0}, i=1,\dots,n  \Big\}
		\]		
		and set $\mathcal{H}(\Gamma):=(\mathrm{Cone}(\Gamma) \setminus \Gamma)\cap \mathbb{N}^d$. An element $ f\in \mathcal{H}(\Gamma)$ is called a pseudo-Frobenius element of $\Gamma$ if $f+\gamma \in \Gamma$ for all $ \gamma \in \Gamma \setminus \{\mathbf{0}\}$. The set of pseudo-Frobenius elements of $\Gamma$ is denoted by $\mathrm{PF}(\Gamma)$.
		In \cite[Theorem 6]{Pseudo}, the authors proved that $\Gamma$ is a MPD-semigroup if and only if $\mathrm{PF}(\Gamma)\neq \emptyset$. In particular, they prove that if $\Gamma$ is a MPD-semigroup then $b \in \Gamma$ is the $\Gamma$-degree of the $(n-2)^{th}$-minimal syzygy of $\mathbb{K}[\Gamma]$ if and only if $b \in \{\mathbf{a}+\sum_{i=1}^{n}\mathbf{a}_i : \mathbf{a}\in \mathrm{PF}(\Gamma)\}$.
		
		Let $(\mathbf{F},\phi)$ be a graded minimal free $R$- resolution of $\mathbb{K}[\Gamma]$, where
		
		$$ \mathbf{F}: 0 \rightarrow R^{\beta_k} \xrightarrow {\phi_k} R^{\beta_{k-1}} \xrightarrow{\phi_{k-1}} \cdots \xrightarrow{\phi_2} R^{\beta_1} \xrightarrow{\phi_0}\rightarrow \mathbb{K}[\Gamma] \rightarrow 0. $$
		
		The elements $s_{i,j} \in \Gamma$ for which $R^{\beta_i}$=$\oplus_{j=1}^{\beta_i}R[-s_{i,j}]$ are called $i$-Betti $\Gamma$-degrees. Denote by $\mathcal{B}_i(\Gamma)$ the set of these $i$-Betti $\Gamma$-degrees for $1 \leq i \leq \mathrm{pd}(\Gamma)$ and let $\mathcal{B}_i(\Gamma)=0$ otherwise, where $\mathrm{pd}(\Gamma)$ is the projective dimension of $\mathbb{K}[\Gamma]$. Note that we allow $\mathcal{B}_i(\Gamma)$
		to contain repeating elements in a nonstandard way for convenience.	
		
		\medspace
		
		\begin{definition}[\cite{Extension}, Definition 2.1]\label{Ext}{\rm
			Given an affine semigroup $\Gamma$ generated minimally by
			the set $\{\mathbf{a}_1,\dots,\mathbf{a}_n\}$, recall that an extension of $\Gamma$ is an affine semigroup denoted by $E$ and generated by $l\mathbf{a}_1,\dots,l\mathbf{a}_n, \mathbf{a}$, where $l$ is a positive integer coprime to a component of $\mathbf{a}=u_1 \mathbf{a}_1+\dots+u_n \mathbf{a}_n$ for some non-negative integers $u_1,\dots,u_n$.}
			
		\end{definition}
		
		\begin{definition}[\cite{Pseudo},Definition 3.1]
			Let $\Gamma$ be an affine MPD-semigroup with $\mathcal{H}(\Gamma)$ is finite. For a fixed term order $\prec$ on $\mathbb{N}^d$, where $\rm{F}(\Gamma)\neq \emptyset$ such that $\rm{F}(\Gamma)_{\prec}=\rm{max_{\prec}}\mathcal{H}(\Gamma)\in \rm{F}(\Gamma)$. If $|\mathrm{PF}(\Gamma)|=1$ and $\mathrm{PF}(\Gamma)=\{\rm{F}(\Gamma)_{\prec}\}$, then $\Gamma$ is called a $ \prec$-symmetric semigroup.
		\end{definition}

		\begin{theorem}
		 Let $\Gamma$ be an affine MPD-semigroup with $\mathcal{H}(\Gamma)$ is finite then extension $E$ of affine semigroup $\Gamma$ is MPD-semigroup and $\mathcal{H}(E)$ is finite. In particular, $\mathrm{PF}(E)=\{lf+(l-1)\mathbf{a}\,| f \in \mathrm{PF}(\Gamma)\}$. Moreover, if $\Gamma$ is $\prec$-symmetric then $E$ is also $\prec$-symmetric.
		 
		\end{theorem}
		\begin{proof}
			Let 	$ \mathbf{F}: 0 \rightarrow R^{\beta_k} \xrightarrow {\phi_k} R^{\beta_{k-1}} \xrightarrow{\phi_{k-1}} \cdots \xrightarrow{\phi_2} R^{\beta_1} \xrightarrow{\phi_0}\rightarrow \mathbb{K}[\Gamma] \rightarrow 0 $
 			be a graded minimal free $R$-resolution of $\mathbb{K}[\Gamma]$. It is clear from the proof of \cite[Lemma 3.4]{Sahin-SIFR} that $\mathcal{B}_i(E)=l\mathcal{B}_i(\Gamma) \cup l[\mathcal{B}_{i-1}(\Gamma)+\mathbf{a}]$ for $1 \leq i \leq \mathrm{pd}(\Gamma)+1$ and $\mathrm{pd}(E)=\mathrm{pd}(\Gamma)+1$. Since $\Gamma$ is MPD-semigroup, therefore $\mathrm{pd}(E)=n-1+1=n$, and so $E$ is MPD semigroup. Now $\mathcal{B}_{n-1}(E)=l\mathcal{B}_{n-1}(\Gamma)\cup [l\mathcal{B}_{n-2}(\Gamma)+\mathbf{a}]=l\mathcal{B}_{n-2}(\Gamma)+l\mathbf{a}$. Therefore, for any $b_{n-2}\in \mathcal{B}_{n-2}(\Gamma)$, we have $lb_{n-2}(\Gamma)+l\mathbf{a}- \sum_{i=1}^{n}l\mathbf{a}_i-\mathbf{a}=lb_{n-2}(\Gamma)-\sum_{i=1}^{n}l\mathbf{a}_i+l\mathbf{a}- \mathbf{a}=lf+(l-1)\mathbf{a}\in \mathrm{PF}(E)$. Hence,
 			$\mathrm{PF}(E)=\{lf+(l-1)\mathbf{a}\,| f \in \mathrm{PF}(\Gamma)\}$. Let $\mathbf{x}\in \mathcal{H}(E)$ then $\mathbf{x}=y_1 \mathbf{a}_1+\dots+y_n\mathbf{a}_n+z\mathbf{a}$, for some $y_i,z\in \mathbb{Q}_{\geq 0}$ but since $\mathbf{a}\in \Gamma$ and one of $y_j$ or $z$ is in $\mathbb{Q}\setminus \mathbb{N}$, we have $\mathbf{x}\in \mathcal{H}(\Gamma)$. Therefore, $\mathcal{H}(E) \subseteq \mathcal{H}(\Gamma)$ is finite. Since $\mathcal{H}(E)$ is finite and $\Gamma$ is $\prec$-symmetric, we have $|\mathrm{PF}(E)|=1 $ and $\mathrm{PF}(E)=\{l\mathrm{F}(\Gamma)_{\prec}+(l-1)\mathbf{a}=\mathrm{F}(E)_{\prec}\}$. Therefore, $E$ is also $\prec$-symmetric.
 		\end{proof}

		\smallskip
		
		\begin{example}[\cite{MPD}]{\rm
			 Let $\Gamma$ be the submonoid of $\mathbb{N}^2$ generated by the columns of the following matrix
			$$A= \begin{pmatrix}
			 	3 &5& 0& 1 &2\\
			 	 0 &0 &1 &3 &3			 	
			 \end{pmatrix}$$
		
		The minimal free resolution $\mathbb{K}[\Gamma]$, as a module over $R=\mathbb{K}[x_1,x_2,x_3,x_4,x_5]$ is given by 
		$$0 \rightarrow R(18,9) \rightarrow R^6 \rightarrow R^{11} \rightarrow R^7 \rightarrow R \rightarrow \mathbb{K}[\Gamma] \rightarrow 0 .$$ 
		In particular, the degrees of minimal generators of the fourth syzygy modules is $(18,9)$. Therefore, $\Gamma$ has one pseudo-Frobenius element $(7,2)$. Consider the extension $E$ of $\Gamma$, which is the submonoid of $\mathbb{N}^2$ generated by the columns of the following matrix
		$$B= \begin{pmatrix}
			6 &10& 0& 2&4 & 6\\
			0 &0 &2 &6 &6	& 9		 	
		\end{pmatrix}$$
			The minimal free resolution $\mathbb{K}[\Gamma]$, as a module over $R=\mathbb{K}[x_1,x_2,x_3,x_4,x_5]$ is given by 
		$$0 \rightarrow R(48,36) \rightarrow R^8 \rightarrow R^{18} \rightarrow R^{17} \rightarrow R^{7} \rightarrow R^{1} \rightarrow \mathbb{K}[E] \rightarrow 0 .$$ 
		The degrees of minimal generators of the fifth syzygy modules is $(48,36)$. Therefore, $E$ has one pseudo-Frobenius element $(48,36)-(28,23)=(20,13)$. Note that $\mathcal{H}(\Gamma)$ is finite, so $\Gamma$ is $\prec$-symmetric and also as $\mathcal{H}(E)\subseteq \mathcal{H}(\Gamma)$, hence $E$ is $\prec$-symmetric.}
		\end{example}
		
		\medspace
		
		\section{Strong Indispensable Minimal Free Resolution}
		Let $\Gamma$ be an affine semigroup, i.e, a finitely generated semigroup which for some $r$ is isomorphic to a subsemigroup of $\mathbb{Z}^r$ containing zero. 
		Suppose that 
		$\Gamma$ is a simplicial affine semigroup, fully embedded in $\mathbb{N}^{d}$, minimally generated by 
		$\{\mathbf{a}_{1}.\dots,\mathbf{a}_{d},\mathbf{a}_{d+1},\dots,\mathbf{a}_{d+r}\}$ with the set 
		of extremal rays $E=\{\mathbf{a}_{1},\dots,\mathbf{a}_{d}\}$. 	The semigroup algebra $\mathbb{K}[\Gamma]$ over a field $\mathbb{K}$ is generated by the 
		monomials $\mathbf{x}^{\mathbf{a}}$, where $\mathbf{a} \in \Gamma$, with maximal ideal 
		$\mathfrak{m}=(\mathbf{x}^{\mathbf{a}_{1}},\dots,\mathbf{x}^{\mathbf{a}_{d+r}})$.
		
		Let $I(\Gamma)$ denote the defining ideal of $\mathbb{K}[\Gamma]$, which is the 
		kernel of the $\mathbb{K}-$algebra homomorphism  
		$\phi:R=\mathbb{K}[z_{1},\dots,z_{d+r}] \rightarrow \mathbb{K}[\mathbf{x}^{\mathbf{a}_{1}},\dots,\mathbf{x}^{\mathbf{a}_{d+r}}]$, 
		such that $\phi(z_{i})=\mathbf{x}^{\mathbf{a}_{i}}$, $i=1,\dots,d+r$. Let 
		us write $\mathbb{K}[\Gamma]\cong A/I(\Gamma)$.  The defining ideal $I(\Gamma)$ 
		is a binomial prime ideal (\cite{Herzog}, Proposition 1.4).

				Let $(\mathbf{F},\phi)$ be a graded minimal free $R$- resolution of $\mathbb{K}[\Gamma]$, where
		
		$$ \mathbf{F}: 0 \rightarrow R^{\beta_k} \xrightarrow {\phi_k} R^{\beta_{k-1}} \xrightarrow{\phi_{k-1}} \cdots \xrightarrow{\phi_2} R^{\beta_1} \xrightarrow{\phi_0}\rightarrow \mathbb{K}[\Gamma] \rightarrow 0. $$
		
		The elements $s_{i,j} \in \Gamma$ for which $R^{\beta_i}$=$\oplus_{j=1}^{\beta_i}R[-s_{i,j}]$ are calledd $i$-Betti $\Gamma$-degrees. Denote by $\mathcal{B}_i(\Gamma)$ the set of these $i$-Betti $\Gamma$-degrees for $1 \leq i \leq \mathrm{pd}(\Gamma)$ and let $\mathcal{B}_i(\Gamma)=0$ otherwise, where $\mathrm{pd}(\Gamma)$ is the projective dimension of $\mathbb{K}[\Gamma]$. Note that we allow $\mathcal{B}_i(\Gamma)$
		to contain repeating elements in a nonstandard way for convenience.	
		
		The resolution $(\mathbf{F},\phi)$ is called \textbf{strongly indispensable} if for any graded minimal resolution $(\mathbf{G},\theta)$, we have an injective complex map $i:(\mathbf{F},\phi) \rightarrow (\mathbf{G},\theta)$. When $(\mathbf{F},\phi)$ is strongly indispensable $\Gamma$ or $\mathbb{K}[\Gamma]$ is said to have a SIFR for short.
		
		\medspace
		
		\subsection{Join of Affine Semigroups}
		
		We start this section with the definition of join of affine semigroups.
		
		\begin{definition}[\cite{Join}]\label{Join}{\rm
				Consider affine semigroups $\Gamma_{1},\Gamma_{2} \subset 
				\mathbb{N}^r$ of dimensions $r_1$ and $r_2$, respectively such 
				that they are 
				minimally generated by 
				the disjoint sets $\mathcal{G}_{1}=\{\mathbf{a}_{1},\dots,\mathbf{a}_{r_{1}},\dots,\mathbf{a}_{r_{1}+n_{1}}\}$ and     $\mathcal{G}_{2}=\{\mathbf{b}_{1},\dots,\mathbf{b}_{r_{2}},\dots,\mathbf{b}_{r_{2}+n_{2}}\}$, with extremal rays $E_{\Gamma_{1}}=\{\mathbf{a}_{1},\dots,\mathbf{a}_{r_{1}}\}$, $E_{\Gamma_{2}}=\{\mathbf{b}_{1},\dots,\mathbf{b}_{r_{2}}\}$ respectively. 
				Assuming that the set $\{\mathbf{a}_{1},\dots,\mathbf{a}_{r_{1}},\mathbf{b}_{1},\dots,\mathbf{b}_{r_{2}}\}$ is linearly independent over $\mathbb{Q}$, the semigroup $\Gamma_{1}\sqcup\Gamma_{2}=\langle\mathcal{G}_{1}\cup\mathcal{G}_{2}\rangle$ is an affine semigroup with the set of  extremal rays $E_{\Gamma_{1}\sqcup\Gamma_{2}}=E_{\Gamma_{1}}\cup E_{\Gamma_{2}}$. Moreover, 
				the set $\mathcal{G}_{1}\cup\mathcal{G}_{2}$ is a minimal generating set of the semigroup 
				$\Gamma_{1}\sqcup\Gamma_{2}$. We call $\Gamma_{1}\sqcup\Gamma_{2}$ the \textit{join} 
				of affine semigroups $\Gamma_{1}$ and $\Gamma_{2}$. Clearly, $\mathrm{dim}(\Gamma_1 \sqcup \Gamma_2)=\mathrm{dim}( \Gamma_1)+\mathrm{dim}( \Gamma_2)$. 
				Let $\Gamma_{1},\dots,\Gamma_{s}$ be affine semigroups in $\mathbb{N}^{r}$, of dimensions $r_1,\dots,r_s$, respectively. 
				If $E_{\Gamma_{1}} \cup \dots \cup E_{\Gamma_{s}}$ is linearly independent set of cardinality $r_1+\dots +r_s=r$, then 
				$\Gamma=\Gamma_1 \sqcup \dots \sqcup \Gamma_s$ is simplicial. Note that a simplicial affine semigroup in $\mathbb{N}^r$ is a join of $r$ semigroups if and only if all elements in its minimal generating set are located on the extremal rays of $\rm{cone}(\Gamma)$. 
				
				In an equivalent statement, one can say that an affine semigroup $\Gamma$ with the minimal generating set $G$ and the set of extremal rays $E$, is a join of two of its sub-semigroups if there exist $G_1,G_2,E_1,E_2$ such that $G=G_1 \cup G_2,\, E=E_1 \cup E_2,\, E_1 \cap E_2=\phi$ and $E_i$ is the set of extremal rays of $\Gamma_i=\langle G_i \rangle$, for $i=1,2$. 
				We call $\Gamma_{1}\sqcup\Gamma_{2}$ the \textit{join} 
				of affine semigroups $\Gamma_{1}$ and $\Gamma_{2}$, denoted by $\Gamma_c$. 
			}
		\end{definition}

		Let $A_{1}=\mathbb{K}[x_{1},\ldots,x_{r_{1}+n_{1}}]$ 
		and $A_{2}=\mathbb{K}[y_{1},\ldots,y_{r_{2}+n_{2}}]$ be polynomial rings 
		with disjoint set of 
		indeterminates $\{x_{1},\ldots , x_{r_{1}+n_{1}}\}$ and $\{y_{1},\dots,y_{r_{2}+n_{2}}\}$. 
		Let $\mathbb{K}[\mathbf{t}]$ be another polynomial ring, where 
		$\mathbf{t}=t_{1},\ldots,t_{r}$. 
		We consider maps $\phi_{\Gamma_{1}}:A_{1}\rightarrow \mathbb{K}[\mathbf{t}]$,  
		defined by $\phi_{\Gamma_{1}}(x_{i})=\mathbf{t}^{\mathbf{a_{i}}}$, $1 \leq i \leq r_{1}+n_{1}$, 
		and 
		$\phi_{\Gamma_{2}}:A_{2}\rightarrow \mathbb{K}[\mathbf{t}]$, defined by 
		$\phi_{\Gamma_{2}}(y_{j})=\mathbf{t}^{\mathbf{b_{j}}}$, $1 \leq j \leq r_{2}+n_{2}$. 
		We write $A_{12}=\mathbb{K}[x_{1},\dots,x_{r_{1}+n_{1}},y_{1},\dots,y_{r_{2}+n_{2}}]$ and 
		consider $\phi_{\Gamma_{1}\sqcup\Gamma_{2}}: A_{12}\rightarrow \mathbb{K}[\mathbf{t}]$, 
		defined by 
		$\phi_{\Gamma_{1}\sqcup\Gamma_{2}}(x_{i})=\mathbf{t}^{\mathbf{a_{i}}}$, 
		$1 \leq i \leq r_{1}+n_{1}$, and 
		$\phi_{\Gamma_{1}\sqcup\Gamma_{2}}(y_{j})=\mathbf{t}^{\mathbf{b_{j}}}$, 
		$1 \leq j \leq r_{2}+n_{2}$. For the sake of simplicity assume $l_1=r_1+n_1$ and $l_2=r_2+n_2$.
		
		\medskip
		
		\begin{theorem}\label{Generator-Concatanation}
			Let $I_{\Gamma_{1}}, I_{\Gamma_{2}}$ be the defining ideals of $\mathbb{K}[\Gamma_{1}]$ and $\mathbb{K}[\Gamma_{2}]$ respectively. Then the defining ideal 
			of $\mathbb{K}[\Gamma_{1}\sqcup\Gamma_{2}]$ is 
			$I_{\Gamma_{1}\sqcup\Gamma_{2}} = I_{\Gamma_{1}}A_{12}+I_{\Gamma_{2}}A_{12}$. The 
			tensor product of minimal graded free resolutions of $\mathbb{K}[\Gamma_{1}]$ and $\mathbb{K}[\Gamma_{2}]$ over $\mathbb{K}$ is a 
			minimal graded free resolution of $\mathbb{K}[\Gamma_{1}\sqcup\Gamma_{2}]$ over $\mathbb{K}$.
		\end{theorem}
		
		\begin{proof}
			See \cite[Theorem 4.2]{Join}.
		\end{proof}
		
		\begin{lemma}\label{Cond-SIFR}
		A minimal graded free resolution of $\mathbb{K}[\Gamma]$ is strongly indispensable if and only if $\pm (\mathbf{b}_i-\mathbf{b}')\notin \Gamma $ for all $\mathbf{b}_i,\mathbf{b}' \in \mathcal{B}_i(\Gamma)$ and for each $1 \leq i \leq \mathrm{pd}(\Gamma)$.
		\begin{proof}
			See \cite[Lemma 2.1]{Sahin-SIFR}.
		\end{proof}

		\end{lemma}
		
		\begin{theorem}\label{Join-SIFR}
			Let $\Gamma_1$ and $\Gamma_2$ are two affine semigroups in $\mathbb{N}^r$. Let $\Gamma=\Gamma_1 \sqcup \Gamma_2$ be a join of affine semigroups $\Gamma_{1},\Gamma_{2}$. Then $\Gamma$ has a SIFR if and only if  $\Gamma_1$ and $\Gamma_2$ have SIFRs. 
		\end{theorem}
		\begin{proof}
			Let $\mathbf{F}_i$ be a graded minimal free resolution of $\mathbb{K}[{\Gamma_i}]$ for $i=1,2$. By Theorem \ref{Generator-Concatanation}, $\mathbf{F}_1\otimes \mathbf{F}_2$ gives the graded minimal free resolution of $\mathbb{K}[{\Gamma}]$. Hence, the proof follows from the following
			$$ 	\big(\mathbf{F}_1 \otimes \mathbf{F}_2 \big)_i= \bigoplus_{p_1+p_2=i} F_{p_1} \otimes F_{p_2}			$$
			since $\Gamma$-degrees of elements in $F_{p_1} \otimes F_{p_2}$ constitutes the set $\mathcal{B}_{p_1}(\Gamma_{1}) +\mathcal{B}_{p_2}(\Gamma_2)$.  Hence, we have
			$\mathcal{B}_i(\Gamma)=\bigcup_{p_1+p_2=i}\big(\mathcal{B}_{p_1}(\Gamma_1) \cup  \mathcal{B}_{p_2}(\Gamma_{2})\big)$.
			\medspace
			
			\textbf{claim(i)}: Fix $j \in \{1,2\}$, and $ \mathbf{b},\mathbf{b'} \in \Gamma_j$. Then $ \mathbf{b}-\mathbf{b'} \in \Gamma $ if and only if $\mathbf{b}-\mathbf{b'} \in \Gamma_1 $ or $\mathbf{b}-\mathbf{b'} \in \Gamma_2 $.\\
			Without loss of generality assume that $j=1$. As $\Gamma_1 \subseteq \Gamma$, $\mathbf{b}-\mathbf{b'} \in \Gamma_1 $ implies $\mathbf{b}-\mathbf{b'} \in \Gamma $. Conversely, take $\mathbf{b},\mathbf{b'} \in \Gamma_1 $ such that $\mathbf{b}-\mathbf{b'} \in \Gamma $. Then $\mathbf{b}-\mathbf{b'} = \gamma_1 +\gamma_2 $ for some $\gamma_1 \in \Gamma_{1}$ and $\gamma_2 \in \Gamma_{2}$.
			Assume that $\gamma_2 \neq 0$. Since $E_1$ and $E_2$ are the set of extremal rays of $\Gamma_1$ and $\Gamma_2$ respectively, hence $\mathbf{b},\, \gamma_1,\,\mathbf{b'},\, \gamma_2$ can be written as rational linear combination of elements of $E_1$ and $E_2$. But in that case we get a non-zero linear combination of elements of $E_1$ and $E_2$ over $\mathbb{Q}$, which is a contradiction of the fact that $\{\mathbf{a}_{1},\dots,\mathbf{a}_{r_{1}},\mathbf{b}_{1},\dots,\mathbf{b}_{r_{2}}\}$ is linearly independent over $\mathbb{Q}$. Therefore $\mathbf{b}-\mathbf{b'} \in \Gamma_1$.

			By Lemma \ref{Cond-SIFR}, the differences between the elements in $\mathcal{B}_i(\Gamma)$ do not belong to $\Gamma$. Let $\mathbf{b}_{i,j}, \mathbf{b}^{'}_{i,j} \in \mathcal{B}_{i}(\Gamma_j)$, for $j=1,2$. Since $\mathcal{B}_i(\Gamma_j) \subset \mathcal{B}_i(\Gamma)$, and
			$\mathbf{b}_{i,j}- \mathbf{b}^{'}_{i,j} \notin \Gamma$. Hence by claim (i) $\mathbf{b}_{i,j}-\mathbf{b}^{'}_{i,j} \notin \Gamma_j$, which means $I_{\Gamma_1}$ and $I_{\Gamma_2}$ have SIFR by Lemma \ref{Cond-SIFR}. Now, if $ \mathbf{b},\mathbf{b}^{'} \in \mathcal{B}_i(\Gamma)$ then $\mathbf{b},\mathbf{b}^{'} \in \mathcal{B}_p(\Gamma_1)+\mathcal{B}_q(\Gamma_2)$, for $p+q=i$. Let $\mathbf{b}=\mathbf{b}_{p,1}+\mathbf{b}_{q,2}$ and $\mathbf{b}^{'}=\mathbf{b}^{'}_{r,1}+\mathbf{b}^{'}_{s,2}$
			with $p+q=i=r+s$. Suppose $\mathbf{b}-\mathbf{b}^{'} \in \Gamma$. Then, $\mathbf{b}-\mathbf{b}^{'}=	\mathbf{b}_{p,1}+\mathbf{b}_{q,2}-	\mathbf{b}^{'}_{r,1}-\mathbf{b}^{'}_{s,2} =\gamma_1+\gamma_2$, for some $\gamma_1 \in \Gamma_{1}$ and $\gamma_2 \in \Gamma_{2}$. Thus, $\mathbf{b}_{p,1}-\mathbf{b}'_{r,1}-\gamma_1=\mathbf{b}_{q,2}-\mathbf{b}'_{s,2}-\gamma_2$, which leads to the contradiction of the fact $\{\mathbf{a}_{1},\dots,\mathbf{a}_{r_{1}},\mathbf{b}_{1},\dots,\mathbf{b}_{r_{2}}\}$ is linearly independent over $\mathbb{Q}$, unless $\mathbf{b}_{p,1}-\mathbf{b}'_{r,1}-\gamma_1=\mathbf{b}_{q,2}-\mathbf{b}'_{s,2}-\gamma_2=0$. Hence, $\mathbf{b}_{p,1}-\mathbf{b}'_{r,1} \in \Gamma_1$ and $\mathbf{b}_{q,2}-\mathbf{b}'_{s,2}\in \Gamma_2$, which is not possible as $\Gamma_1$ and $\Gamma_2$ have SIFRs. Hence, $\mathbf{b}-\mathbf{b}^{'} \notin \Gamma$, and by Lemma \ref{Cond-SIFR}, $\Gamma$ has a strongly indispensable free resolution.
			\end{proof}
	
		\medspace

		\bibliographystyle{amsalpha}

\begin{thebibliography}{A}
		
		\bibitem{Arslan}
		Arslan, Feza; Mete, Pinar; \c{S}ahin, Mesut.
		Gluing and Hilbert functions of monomial curves.
		Proc. Amer. Math. Soc.137(2009), no.7, 2225–2232.
		
		\bibitem{Arslan-2}
		Arslan, Feza; Mete, Pinar.
		Hilbert functions of Gorenstein monomial curves. Proc. Amer. Math. Soc.135(2007), no.7, 1993–2002.
		
		\bibitem{Pseudo}
		Bhardwaj, Om Prakash; Goel, Kriti; Sengupta, Indranath. 
		Affine semigroups of maximal projective dimension.(English summary)Collect. Math.74(2023), no.3, 703–727.
				
		\bibitem{SIFR}
		Charalambous, Hara; Thoma, Apostolos.
		On the generalized Scarf complex of lattice ideals. J. Algebra 323(2010), no.5, 1197–1211.
		
		\bibitem{MPD}
		Garc\'{i}a-Garc\'{i}a, J. I.; Ojeda, I.; Rosales, J. C.; Vigneron-Tenorio, A.
		On pseudo-Frobenius elements of submonoids of $\mathbb{N}^d$. Collect. Math.71(2020), no.1, 189–204.
		
		
		\bibitem{Macaulay}
		Grayson, D.R.; Stillman, M.E. Macaulay 2, a software system for research in
		algebraic geometry. Available at http://www.math.uiuc.edu/Macaulay2.
		
		\bibitem{SMR}
		Gimenez, Philippe; Srinivasan, Hema.
		The structure of the minimal free resolution of semigroup rings obtained by gluing. J. Pure Appl. Algebra 223(2019), no.4, 1411–1426.		
		
		\bibitem{Herzog}
		Herzog, J\"{u}rgen. Generators and relations of abelian semigroups and semigroup rings. Manuscripta
		Math.3 (1970), 175-193.
		
		\bibitem{CMC}
		Herzog, J\"{u}rgen; Stamate, Dumitru I.
		Cohen-Macaulay criteria for projective monomial curves via Gr\"{o}bner bases. Acta Math. Vietnam.44(2019), no.1, 51–64.
		
		
		
		
				
		\bibitem{Kunz}
		Kunz, Ernst. The value-semigroup of a one-dimensional Gorenstein ring. Proc. Amer. Math. Soc.25(1970), 748–751.
		
		\bibitem{ACM}
		Patil, D. P.; Roberts, L. G. 
		Hilbert functions of monomial curves. J. Pure Appl. Algebra183(2003), no.1-3, 275–292.
		
		\bibitem{Rosales}
		Rosales, J. C.	On presentations of subsemigroups of $\mathbb{N}^n$. Semigroup Forum55(1997), no.2, 152–159.
		
		

		
		
				
		\bibitem{Join}
		 Saha, Joydip; Sengupta, Indranath; Srivastava, Pranjal. Join of affine semigroups. Communications in Algebra(2023), DOI: 10.1080/00927872.2023.2266836.
				
		\bibitem{Saha-Gluing}
		Saha, Joydip; Sengupta, Indranath; Srivastava, Pranjal.
		Betti sequence of the projective closure of affine monomial curves. J. Symbolic Comput.119(2023), 101–111.
		
		
		

	
	
		
		\bibitem{Sahin-SIFR}
		\c{S}ahin, Mesut; Stella, Leah Gold.
		Gluing semigroups and strongly indispensable free resolutions. Internat. J. Algebra Comput.29(2019), no.2, 263–278.
		
		\bibitem{Extension}
		\c{Ş}ahin, Mesut.
		Extensions of toric varieties. Electron. J. Combin.18(2011), no.1, Paper 93, 10 pp.
		
		
				
		\bibitem{Betti}
		Stamate, Dumitru I.
		Betti numbers for numerical semigroup rings. Multigraded algebra and applications, 133–157.
		Springer Proc. Math. Stat., 238
		
		
		\bibitem{Watanabe}
		Watanabe, Keiichi. Some examples of one dimensional Gorenstein domains. Nagoya Math. J.49(1973), 101–109.
	\end{thebibliography}

	\end{document}